\documentclass[aop, preprint]{imsart}

\RequirePackage{amsthm,amsmath,amsfonts,amssymb,bbm}
\RequirePackage{eurosym, graphicx}
\RequirePackage[numbers]{natbib}
\RequirePackage[colorlinks,citecolor=blue,urlcolor=blue]{hyperref}

\startlocaldefs
\numberwithin{equation}{section}
\theoremstyle{definition}
\newtheorem{definition}{Definition}
\theoremstyle{remark}
\newtheorem{rem}{Remark}
\theoremstyle{plain}
\newtheorem{assumption}{Assumption}
\newtheorem{theorem}{Theorem}
\newtheorem{lemma}{Lemma}

\newcommand{\lfrac}[2]{\genfrac{}{}{0pt}{}{#1}{#2}}
\endlocaldefs
{\catcode`\@=11 \uccode`9=`\l \uccode`8=`\o %
	\uppercase{\gdef\striplong@#1#2#3#4\relax{%
			\ifx9#2\ifx8#3\@xp\@xp\@xp\@xp\@xp\@xp\@xp\zap@to@space\fi\fi}}}
\begin{document}

\begin{frontmatter}

\title{On rate of convergence in non-central limit theorems}
\runtitle{Rate  of convergence to Hermite-type distribution}
\begin{aug}
\author{\fnms{Vo} \snm{Anh}\thanksref[1]{t1}\ead[label=e1]{v.anh@qut.edu.au}},
\author{\fnms{Nikolai} \snm{Leonenko}\thanksref[1,2]{t1,t2}\ead[label=e2]{LeonenkoN@cardiff.ac.uk}},
\author{\fnms{Andriy} \snm{Olenko}\corref{}\thanksref[1,3]{t1,t3}\ead[label=e3]{a.olenko@latrobe.edu.au}}
\and
\author{\fnms{Volodymyr} \snm{Vaskovych}\ead[label=e4]{vaskovych.v@students.latrobe.edu.au}}

\thankstext[1]{t1}{Supported in part under Australian Research Council's
Discovery Projects funding scheme (project number DP160101366)}
\thankstext[2]{t2}{Supported in part by project MTM2012-32674 (co-funded with FEDER) of the DGI, MINECO,  and under Cardiff Incoming Visiting Fellowship Scheme and International Collaboration Seedcorn Fund}
\thankstext[3]{t3}{Supported in part by the La Trobe University DRP Grant in Mathematical and Computing Sciences}
\runauthor{Anh, Leonenko, Olenko and Vaskovych}

\affiliation{Queensland University of Technology, Cardiff University\\ and La Trobe University}

\address{Vo Anh\\
School of Mathematical Sciences\\
Queensland University of Technology\\
Brisbane, Queensland, 4001\\
Australia\\
\printead{e1}}

\address{N. Leonenko\\
School of Mathematics\\ 
Cardiff University\\ 
Senghennydd Road, Cardiff CF24 4AG\\ 
United Kingdom\\
\printead{e2}}	
		
\address{A. Olenko\\
V. Vaskovych\\
Department of Mathematics and Statistics\\
La Trobe University\\
Melbourne, Victoria, 3086\\
Australia\\
\printead{e3}\\
\phantom{E-mail:\ }\printead*{e4}}

\end{aug}

\begin{abstract}
	The main result of this paper is the rate of convergence to Hermite-type distributions in non-central limit theorems. To the best of our knowledge, this is the first result in the literature on rates of convergence of functionals of random fields to Hermite-type distributions with ranks greater than~2. The results were obtained under rather general assumptions on the spectral densities of random fields. These assumptions are even weaker than in the known convergence results for the case of Rosenblatt distributions.  Additionally, L\'{e}vy concentration functions for Hermite-type distributions were investigated.
\end{abstract}

\begin{keyword}[class=MSC]
\kwd[Primary ]{60G60}
\kwd[; secondary ]{60F05}
\kwd{60G12}
\end{keyword}

\begin{keyword}
\kwd{Rate of convergence}
\kwd{Non-central limit theorems}
\kwd{Random field}
\kwd{Long-range dependence}
\kwd{Hermite-type distribution}
\end{keyword}

\end{frontmatter}

\section{Introduction}

This research will focus on the rate of convergence of local functionals of real-valued homogeneous random fields
 with long-range dependence. Non-linear integral functionals on
bounded sets of $\mathbb{R}^{d}$ are studied. These functionals are
important statistical tools in various fields of application, for
example, image analysis, cosmology, finance, and geology. It was shown in \cite{Dob}, 
\cite{Taq1} and \cite{Taq2} that these functionals can produce non-Gaussian
limits and require normalizing coefficients different from those in central
limit theorems.

Since many modern statistical models are now designed to deal with
non-Gaussian data, non-central limit theory is gaining more and more
popularity. Some novel results using different models and asymptotic
distributions were obtained during the past few years, see \cite{Main}, \cite%
{Bre1}, \cite{marpec}, \cite{pec}, \cite{Taq1} and references therein.
Despite such development of the asymptotic theory, only a few of the studies
obtained the rate of convergence, especially in the non-central case.

There are two popular approaches to investigate the rate of convergence in
the literature: the direct probability approach \cite{Main}, \cite{Leon}, and
the Stein-Malliavin method introduced in \cite{NourPec1}.

As the name suggests, the Stein-Malliavin method combines Malliavin calculus
and Stein's method. The main strength of this approach is that it does not
use any restrictions on the moments of order higher than four (see, for
example, \cite{NourPec1}) and even three in some cases (see \cite{NeufViens}%
). For a more detailed description of the method, the reader is referred to 
\cite{NourPec1}. At this moment, the Stein-Malliavin approach is well
developed for stochastic processes. However,
many problems concerning non-central limit theorems for random fields remain
unsolved. The full list of the already solved problems can be found in \cite%
{NourWeb}.

One of the first papers which obtained the rate of convergence in the
central limit theorem using the Stein-Malliavin approach was \cite{NourPec1}%
. The case of stochastic processes was considered. Further refinement of
these results can be found in \cite{NourPec2}, where optimal Berry-Esseen
bounds for the normal approximation of functionals of Gaussian fields are
shown. However, it is known that numerous functionals do not converge to the
Gaussian distribution. The conditions to obtain the Gaussian asymptotics can
be found in so-called Breuer-Major theorems, see \cite{Arc} and \cite{Surg}.
These results are based on the method of cumulants and diagram formulae.
Using the Stein-Malliavin approach, \cite{NourPec3} derived a version of a
quantitative Breuer-Major theorem that contains a stronger version of the
results in \cite{Arc} and \cite{Surg}. The rate of convergence for
Wasserstein topology was found and an upper bound for the Kolmogorov
distance was given as a relationship between the Kolmogorov and Wasserstein
distances. In \cite{KimPark} the authors directly derived the upper-bound
for the Kolmogorov distance in the same quantitative Breuer-Major theorem as
in \cite{NourPec3} and showed that this bound is better than the known
bounds in the literature, since it converges to zero faster. The results
described above are the most general results currently known concerning the
rate of convergence in the central limit theorem using the Stein-Malliavin
approach.

Related to \cite{NourPec3} is the work \cite{Pham} where,
using the same arguments, the author found the rate of convergence for the
central limit theorem of sojourn times of Gaussian fields. Similar results
for the Kolmogorov distance were obtained in \cite{KimPark}.

Concerning non-central limit theorems, only partial results have been found.
It is known from \cite{Dav},\cite{Surg} and \cite{Taq1} that, depending on
the value of the Hurst parameter, functionals of fractional Brownian motion
can converge either to the standard Gaussian distribution or a Hermite-type
distribution. This idea was used in \cite{Bre1} and \cite{Bre2} to obtain
the first rates of convergence in non-central limit theorems using the
Stein-Malliavin method. Similar to the case of central limit theorems, these
results were obtained for stochastic processes. In \cite{Bre2} fractional
Brownian motion was considered, and rates of convergence for both Gaussian
and Hermite-type asymptotic distributions were given. Furthermore all the
results of \cite{Bre2} were refined in \cite{Bre1} for the case of the
fractional Brownian sheet as an initial random element. It makes \cite{Bre2}
the only known work that uses the Stein-Malliavin method to provide the rate
of convergence of some local functionals of random fields with long-range
dependence.

Separately stands \cite{ArrAz}. This work followed a new approach based only
on Stein's method without Malliavin calculus. The authors worked with
Wasserstein-2 metrics and showed the rate of convergence of quadratic
functionals of i.i.d. Gaussian variables. It is one of the convergence
results which can't be obtained using the regular Stein-Malliavin method 
\cite{ArrAz}. However, we are not aware of extensions of these results to the multi-dimensional and non-Gaussian cases.

The classical probability approach  employs direct probability methods to find the
rate of convergence. Its main advantage over the other methods is that it
directly uses the correlation functions and spectral densities of the
involved random fields. Therefore, asymptotic results can be explicitly
obtained for wide classes of random fields using slowly varying functions.
Using this approach, the first rate of convergence in the central limit
theorem for Gaussian fields was obtained in \cite{Leon}. In the following
years, some other results were obtained, but all of them studied the
convergence to the Gaussian distribution.

As for convergence to non-Gaussian distributions, the only known result
using the classical probability approach is \cite{Main}. For functionals of
Hermite rank-2 polynomials of long-range dependent Gaussian fields, it
investigated the rate of convergence in the Kolmogorov metric of these
functionals to the Rosenblatt-type distribution. In this paper, we
generalize these results to some classes of Hermite-type distributions. It
is worth mentioning that our present results are obtained under more natural
and much weaker assumptions on the spectral densities than those in \cite%
{Main}. These quite general assumptions allow to consider various new
asymptotic scenarios even for the Rosenblatt-type case in \cite{Main}.

It's also worth mentioning that in the known Stein-Malliavin results, the rate of convergence was obtained only for a leading term or a fixed number of chaoses in the Wiener chaos expansion. However, while other expansion terms in higher level Wiener chaoses do not change the asymptotic distribution, they can substantially contribute to the rate of convergence. The method proposed in this manuscript takes into account all terms in the Wiener chaos expansion to derive rates of convergence.

It is well known, see \cite{Dav, NourNu, Shi}, that the probability distributions of
Hermite-type random variables are absolutely continuous. In this paper we 
investigate some fine properties of these distributions required to derive rates of convergence. Specifically, we discuss the cases of bounded probability 
density functions of Hermite-type random variables. Using the 
method proposed in \cite{NourPol}, we derive the anti-concentration inequality that 
can be applied to estimate the L\'{e}vy concentration function of Hermite-type random 
variables.

The article is organized as follows. In Section \ref{sec2} we recall some
basic definitions and formulae of the spectral theory of random fields. The
main assumptions and auxiliary results are stated in Section \ref{sec3}. In
Section \ref{sec4} we discuss some fine properties of Hermite-type
distributions. Section \ref{sec5} provides the results concerning the rate
of convergence. Discussions and conclusions are presented in Section \ref{sec6}.

\section{Notations}

\label{sec2}

In what follows $\left| \cdot \right| $ and $\left\| \cdot \right\| $ denote
the Lebesgue measure and the Euclidean distance in~$\mathbb{R}^{d},$
respectively. We use the symbols $C$ and $\delta$ to denote constants which
are not important for our exposition. Moreover, the same symbol may be used
for different constants appearing in the same proof.

We consider a measurable mean-square continuous zero-mean homogeneous
isotropic real-valued random field $\eta (x),\ x\in \mathbb{R}^{d},$ defined
on a probability space $(\Omega ,\mathcal{F},P),$ with the covariance function 
\begin{equation*}
	\text{\rm{B}}(r):=\mathrm{Cov}\left( \eta (x),\eta (y)\right)
	=\int_{0}^{\infty }Y_{d}(rz)\,\mathrm{d}\Phi (z),\ x,y\in \mathbb{R}^{d},
\end{equation*}%
where $r:=\left\Vert x-y\right\Vert ,$ $\Phi (\cdot )$ is the isotropic
spectral measure, the function $Y_{d}(\cdot )$ is defined by 
\begin{equation*}
	Y_{d}(z):=2^{(d-2)/2}\Gamma \left( \frac{d}{2}\right) \ J_{(d-2)/2}(z)\
	z^{(2-d)/2},\quad z\geq 0,
\end{equation*}%
$J_{(d-2)/2}(\cdot )$ being the Bessel function of the first kind of order $%
(d-2)/2.$

\begin{definition}
	The random field $\eta (x),$ $x\in \mathbb{R}^{d},$ as defined above is said
	to possess an absolutely continuous spectrum if there exists a function $%
	f(\cdot )$ such that 
	\begin{equation*}
		\Phi (z)=2\pi ^{d/2}\Gamma ^{-1}\left( d/2\right) \int_{0}^{z}u^{d-1}f(u)\,%
		\mathrm{d}u,\quad z\geq 0,\quad u^{d-1}f(u)\in L_{1}(\mathbb{R}_{+}).
	\end{equation*}%
	The function $f(\cdot )$ is called the isotropic spectral density function
	of the field~$\eta (x).$ In this case, the field $\eta (x)$ with an absolutely continuous spectrum has the isonormal spectral representation 
	\begin{equation*}
		\eta (x)=\int_{\mathbb{R}^d}e^{i(\lambda ,x)}\sqrt{f\left( \left\| \lambda
			\right\| \right) }W(\mathrm{d}\lambda ),
	\end{equation*}
	where $W(\cdot )$ is the complex Gaussian white noise random measure on $%
	\mathbb{R}^d.$
\end{definition}

Consider a Jordan-measurable bounded set $\Delta \subset \mathbb{R}%
^{d}$ such that $\left\vert \Delta \right\vert >0$ and $\Delta $ contains
the origin in its interior. Let $\Delta (r),r>0,$ be the homothetic image of
the set $\Delta ,$ with the centre of homothety at the origin and
the coefficient $r>0,$ that is $\left\vert \Delta (r)\right\vert
=r^{d}\left\vert \Delta \right\vert .$

Consider the uniform distribution on $\Delta (r)$ with the probability
density function (pdf) $r^{-d}\left|\Delta \right|^{-1} \chi_{%
	\raisebox{-3pt}{\scriptsize $\Delta(r)$}}(x),$ $x\in\mathbb{R}^{d},$ where $%
\chi_{\raisebox{-3pt}{\scriptsize $A$}}(\cdot)$ is the indicator function of
a set $A.$

\begin{definition}
	Let $U$ and $V$ be two random vectors which are independent and uniformly
	distributed inside the set $\Delta (r).$ We denote by $\psi _{\Delta (r)}(z
	),$ $z \geq 0,$ the pdf of the distance $\left\| U-V\right\| $ between $U$
	and $V.$
\end{definition}

Note that $\psi _{\Delta (r)}(z)=0$ if $z>diam\left\{ \Delta (r)\right\} .$
Using the above notations, one can obtain the representation 
\begin{equation*}
	\int_{\Delta (r)}\int_{\Delta (r)}\varUpsilon(\left\Vert x-y\right\Vert )\,%
	\mathrm{d}x\,\mathrm{d}y=\left\vert \Delta \right\vert ^{2}r^{2d}\mathbf{E}\ %
	\varUpsilon(\left\Vert U-V\right\Vert )
\end{equation*}%
\begin{equation}
	=\left\vert \Delta \right\vert ^{2}r^{2d}\int_{0}^{diam\left\{ \Delta
		(r)\right\} }\varUpsilon(z)\ \psi _{\Delta (r)}(z)\,\mathrm{d}z,
	\label{dint}
\end{equation}%
where $\varUpsilon(\cdot )$ is an integrable Borel function.

\begin{rem}
	If $\Delta (r)$ is the ball $v(r):=\{x\in \mathbb{R }^{d}:\left\| x\right\|
	<r\},$ then
	
	\begin{equation*}
		\psi _{v(r)}(z)=d\,r^{-d} z^{d-1}I_{1-(z/2r)^{2}}\left( \frac{d+1}{2} ,\frac{%
			1}{2}\right) ,\quad 0\leq z \leq 2r,
	\end{equation*}
	where 
	\begin{equation*}
		I_{\mu }(p,q):=\frac{\Gamma (p+q)}{\Gamma (p)\ \Gamma (q)}\int_{0}^{ \mu
		}u^{p-1}(1-u)^{q-1}\,\mathrm{d}u,\quad \mu \in ( 0,1],\quad p>0,\ q>0 ,
	\end{equation*}
	is the incomplete beta function, see \cite{iv}.
\end{rem}

\begin{rem}
	Let $H_{k}(u)$, $k\geq 0$, $u\in \mathbb{R}$, be the Hermite polynomials,
	see~\cite{pec}. If $(\xi _{1},\ldots ,\xi _{2p})$ is a $2p$-dimensional
	zero-mean Gaussian vector with 
	\begin{equation*}
		\mathbf{E}\xi _{j}\xi _{k}=%
		\begin{cases}
			1, & \mbox{if }k=j, \\ 
			r_{j}, & \mbox{if }k=j+p\ \mbox{and }1\leq j\leq p, \\ 
			0, & \mbox{otherwise,}%
		\end{cases}%
	\end{equation*}%
	then 
	\begin{equation*}
		\mathbf{E}\ \prod_{j=1}^{p}H_{k_{j}}(\xi _{j})H_{m_{j}}(\xi
		_{j+p})=\prod_{j=1}^{p}\delta _{k_{j}}^{m_{j}}\ k_{j}!\ r_{j}^{k_{j}}.
	\end{equation*}
\end{rem}

The Hermite polynomials form a complete orthogonal system in the\\ Hilbert space 
\begin{equation*}
	{L}_{2}(\mathbb{R},\phi (w)\,dw)=\left\{ G:\int_{\mathbb{R}}G^{2}(w)\phi
	(w)\,\mathrm{d}w<\infty \right\} ,\quad \phi (w):=\frac{1}{\sqrt{2\pi }}e^{-%
		\frac{w^{2}}{2}}.
\end{equation*}

An arbitrary function $G(w)\in {L}_{2}(\mathbb{R},\phi(w )\, dw)$ admits the
mean-square convergent expansion 
\begin{equation}  \label{herm}
	G(w)=\sum_{j=0}^{\infty }\frac{C_{j}H_{j}(w) }{j !}, \qquad C_{j }:=\int_{%
		\mathbb{R}}G(w)H_{j }(w)\phi ( w )\,\mathrm{d}w.
\end{equation}

By Parseval's identity 
\begin{equation}  \label{par}
	\sum_{j=0}^\infty\frac{C_{j}^{2}}{j !} =\int_{\mathbb{R}}G^2(w) \phi ( w )\,%
	\mathrm{d}w.
\end{equation}

\begin{definition}\label{rank} \rm{\cite{Taq1}} Let $G(w)\in {L}_{2}(\mathbb{R},\phi (w)\,dw)$ and
	assume there exists an integer $\kappa \in \mathbb{N}$ such that $C_{j}=0$, for all $%
	0\leq j\leq \kappa -1,$ but $C_{\kappa }\neq 0.$ Then $\kappa $ is called
	the Hermite rank of $G(\cdot )$ and is denoted by $H\mbox{rank}\,G.$
\end{definition}

\begin{definition}
	\rm{\cite{bin}} A measurable function $L:(0,\infty )\rightarrow
	(0,\infty )$ is said to be slowly varying at infinity if for all $t>0,$%
	\begin{equation*}
		\lim\limits_{r \rightarrow \infty }\frac{L(r t)}{L(r )}=1.
	\end{equation*}
\end{definition}

By the representation theorem~\cite[Theorem 1.3.1]{bin}, there exists $C > 0$
such that for all $r \ge C$ the function $L(\cdot)$ can be written in the
form 
\begin{equation*}
	L(r) = \exp \left(\zeta_1(r) + \int_C^r \frac{\zeta_2(u)}{u}\,\mathrm{d}u
	\right),
\end{equation*}
where $\zeta_1(\cdot)$ and $\zeta_2(\cdot)$ are such measurable and bounded
functions that \\
$\zeta_2(r)\to 0$ and $\zeta_1(r)\to C_0$ $(|C_0|<\infty),$
when $r\to \infty.$

If $L(\cdot )$ varies slowly, then $r^{a}L(r)\rightarrow \infty ,$ $%
r^{-a}L(r)\rightarrow 0$ for an arbitrary $a>0$ when $r\rightarrow \infty ,$
see Proposition 1.3.6 \cite{bin}.

\begin{definition}
	\cite{bin} A measurable function $g:(0,\infty )\rightarrow (0,\infty )$ is
	said to be regularly varying at infinity, denoted $g(\cdot )\in R_{\tau }$,
	if there exists $\tau $ such that, for all $t>0,$ it holds that 
	\begin{equation*}
		\lim\limits_{r \rightarrow \infty }\frac{g(r t)}{g(r )}%
		=t^{\tau }.
	\end{equation*}
\end{definition}

\begin{definition}
	\label{sr2}\cite{bin} Let	$g:(0,\infty )\rightarrow (0,\infty )$ be a measurable function and $g(x)\rightarrow 0$
	as $x\rightarrow 0$. Then a slowly varying function $L(\cdot )$ is said to be slowly varying with
	remainder of type 2, or that it belongs to the class SR2, if 
	\begin{equation*}
		\forall x >1:\quad \frac{L(r x)}{L(r)}-1\sim k(x
		)g(r),\quad r\rightarrow \infty ,
	\end{equation*}%
	for some function $k(\cdot )$.
	
	If there exists $x $ such that $k(x )\neq 0$ and $k(x \mu
	)\neq k(\mu )$ for all $\mu $, then $g(\cdot )\in R_{\tau }$ for some $\tau
	\leq 0$ and $k(x )=ch_{\tau }(x )$, where 
	\begin{equation}
		h_{\tau }(x )=%
		\begin{cases}
			\ln (x ),\quad if\,\tau =0, \\ 
			\frac{x ^{\tau }-1}{\tau },\quad \,if\,\tau \neq 0.%
		\end{cases}
		\label{htau}
	\end{equation}
\end{definition}

\section{Assumptions and auxiliary results}

\label{sec3}

In this section, we list the main assumptions and some auxiliary results
from \cite{mink} which will be used to obtain the rate of convergence in
non-central limit theorems.

\begin{assumption}
	\label{ass1} {\ Let $\eta (x),$ $x\in \mathbb{R}^{d}$, be a homogeneous
		isotropic Gaussian random field with $\mathbf{E}\eta (x)=0$ and a covariance
		function $B(x)$ such that }
	
	\begin{equation*}
		B(0)=1,\quad B(x)=\mathbf{E}\eta (0) \eta (x)= \left\Vert x\right\Vert
		^{-\alpha }L(\left\Vert x\right\Vert ),
	\end{equation*}
	where $L(\left\Vert \cdot\right\Vert )$ is a function slowly varying at
	infinity.
\end{assumption}

In this paper we restrict our consideration to $\alpha \in (0,d/\kappa ),$
where $\kappa $ is the Hermite rank in Definition~\ref{rank}. For
such $\alpha $ the covariance function $B(x)$ satisfying Assumption~\ref%
{ass1} is not integrable, which corresponds to the case of long-range
dependence.

Let us denote 
\begin{equation*}
	K_r :=\int_{\Delta(r)}G(\eta (x))\,\mathrm{d}x \quad \mbox{and} \quad
	K_{r,\kappa} :=\frac{C_\kappa}{\kappa !} \int_{\Delta(r)}H_\kappa (\eta
	(x))\,\mathrm{d}x,
\end{equation*}
where $C_\kappa$ is defined by (\ref{herm}).

\begin{theorem}
	{\rm{\cite{mink}}} Suppose that $\eta (x),$ $x\in \mathbb{R}^{d},$
	satisfies Assumption~{\rm\ref{ass1}} and $H\mathrm{rank}\,G=\kappa \in \mathbb{N}.
	$ If at least one of the following random variables
	\begin{equation*}
		\frac{K_{r}}{\sqrt{\mathbf{Var}\text{ }K_{r}}},\quad \frac{K_{r}}{\sqrt{%
				\mathbf{Var}\ K_{r,\kappa }}}\quad \mbox{and}\quad \frac{K_{r,\kappa }}{%
			\sqrt{\mathbf{Var}\ K_{r,\kappa }}},
	\end{equation*}%
	has a limit distribution, then the limit distributions of the other random variables also exist and
	they coincide when $r\rightarrow \infty .$
\end{theorem}

\begin{assumption}
	\label{ass2} The random field $\eta(x),$ $x \in \mathbb{R}^d,$ has the
	spectral density
	
	\begin{equation*}
		f(\left\| \lambda \right\| )= c_2(d,\alpha )\left\| \lambda \right\|
		^{\alpha -d}L\left( \frac 1{\left\| \lambda \right\| }\right),
	\end{equation*}
	where 
	\begin{equation*}
		c_2(d,\alpha ):=\frac{\Gamma \left( \frac{d-\alpha }2\right) }{2^\alpha \pi
			^{d/2}\Gamma \left( \frac \alpha 2\right) },
	\end{equation*}
	and $L(\left\Vert \cdot\right\Vert )$ is a locally bounded function which is
	slowly varying at infinity and satisfies for sufficiently large $r$ the
	condition 
	\begin{equation}  \label{gr}
		\left|1-\frac{L(tr)}{L(r)}\right|\le C\,g(r)h_{\tau}(t),\ t\ge 1,
	\end{equation}
	where $g(\cdot) \in R_{\tau} , \tau \le 0$, such that $g(x) \to 0, \ x \to
	\infty$, and $h_{\tau}(t)$ is defined by~\rm{(\ref{htau}).}
\end{assumption}

\begin{rem}
In applied statistical analysis of long-range dependent models researchers often assume an equivalence of Assumptions~1 and 2. However, this claim is not true in general, see \cite{gub, leoole}. This is the main reason of using both assumptions to formulate the most general result in Theorem~\ref{th5}. However, in various specific cases just one of the assumptions may be sufficient. For example, if $f(\cdot)$ is decreasing in a neighbourhood of zero and continuous for all $\lambda \neq 0,$ then by Tauberian Theorem~4 \cite{leoole} both assumptions are simultaneously satisfied. A detailed discussion of relations between Assumption~\ref{ass1} and \ref{ass2} and various examples can be found in \cite{leoole,ole}. Some important models  used in spatial data analysis and geostatistics that simultaneously satisfy Assumptions~1 and 2 are Cauchy and Linnik's fields, see~\cite{Main}. Their covariance functions are of the form $B\left( x\right) =\left( 1+\left\Vert x\right\Vert^{{}\sigma}\right)
^{-\theta},$ $\sigma \in \left( 0,2\right],$ $\theta >0.$  Exact expressions for their spectral densities in the form required by Assumption 2 are provided in Section 5  \cite{Main}.
\end{rem}

The remarks below clarify condition (\ref{gr}) and compare it with the
assumptions used in \rmfamily\cite{Main}.

\begin{rem}
	\label{rem1} This assumption implies weaker restrictions on the spectral
	density than the ones used in \rm{\cite{Main}}. Slowly varying functions
	in Assumption \ref{ass2} can tend to infinity or zero. This is an
	improvement over \cite{Main} where slowly varying functions were assumed to
	converge to a constant. For example, a function that satisfies this
	assumption, but would not fit that of \rm{\cite{Main}, }is $\ln (\cdot )$.
\end{rem}

\begin{rem}
	If we consider the equivalence in Definition \ref{sr2} in the uniform sense,
	then all the functions in the class SR2 satisfy condition (\ref{gr}). If we
	consider this equivalence in the non-uniform sense, then there are functions
	from SR2 that do not satisfy (\ref{gr}). An example of such functions is $%
	\ln ^{2}(\cdot )$.
\end{rem}

\begin{rem}
	\label{rem3} By Corollary 3.12.3 \rm{\cite{bin} }for $\tau \neq 0$ the
	slowly varying function $L(\cdot )$ in Assumption \ref{ass2} can be
	represented as 
	\begin{equation*}
		L(x)=C\left( 1+c\tau ^{-1}g(x)+o(g(x))\right) .
	\end{equation*}
	As we can see $L(\cdot )$ converges to some constant as $x$ goes to infinity.
	This makes the case $\tau =0$ particularly interesting as this is the only
	case when a slowly varying function with remainder can tend to infinity or
	zero.
\end{rem}

\begin{lemma}
	~\label{lem0} If $L$ satisfies {\rm{(\ref{gr})}}, then for any $%
	k\in \mathbb{N}$, $\delta >0$, and sufficiently large $r$ 
	\begin{equation*}
		\left\vert 1-\frac{L^{k/2}(tr)}{L^{k/2}(r)}\right\vert \leq C\,g(r)h_{\tau
		}(t)t^{\delta },\ t\geq 1.
	\end{equation*}
\end{lemma}

\begin{proof}
	Applying the mean value theorem to the function $f(u) = u^{n},$ $n\in \mathbb{R}$, on $\mathbf{A} = [\min(1,u),\max(1,u)]$ we obtain the inequality
	\begin{equation*}
		1 - x^n = n\theta^{n-1}(1 - x) \le n(1-x)\max(1,x^{n-1}),\, \theta \in \mathbf{A}.
	\end{equation*}
	
	Now, using this inequality for $x = \frac{L(tr)}{L(r)}$ and $n = k/2$ we get 
	\begin{equation}\label{mideq}
		\left|1-\frac{L^{k/2}(tr)}{L^{k/2}(r)}\right| \le n\left|1-\frac{L(tr)}{L(r)}\right|\max\left(1,\left(\frac{L(tr)}{L(r)}\right)^{\frac{k}{2}-1}\right).
	\end{equation}
	By Theorem~1.5.6 \cite{bin} we know there exists $c>0$ such that for any $\delta_1>0$
	\[\frac{L(tr)}{L(r)}\le C\cdot t^{\delta_1},\ t\ge 1.\]
	Applying this result and condition~{\rm(\ref{gr})} to \rm(\ref{mideq}) we get
	\begin{equation*}
		\left|1-\frac{L^{k/2}(tr)}{L^{k/2}(r)}\right|\le
		C\,g(r)h_{\tau}(t)\max\left(1,t^{\delta_1\left(\frac{k}{2}-1\right)}\right)\le C\,g(r)h_{\tau}(t)t^{\delta},\ t\ge 1.
	\end{equation*}
\end{proof}

Let us denote the Fourier transform of the indicator function of the set $%
\Delta$ by 
\begin{equation*}
	K_{\Delta}(x):=\int_{\Delta }e^{i(x,u)} \,\mathrm{d}u,\quad x\in\mathbb{R}^d.
\end{equation*}

\begin{lemma}
	\label{finint} {\rm{\cite{mink}}} If $t_1,...,t_\kappa,$ $\kappa\ge 1,$
	are positive constants such that it holds $\sum_{i=1}^\kappa t_i <d,$ then 
	\begin{equation*}
		\int_{\mathbb{R}^{d\kappa}}|K_{\Delta}(\lambda _1+\cdots +\lambda
		_\kappa)|^2 \frac{\mathrm{d}\lambda _1\ldots \,\mathrm{d}\lambda _\kappa}{%
			\left\| \lambda _1\right\| ^{d-t_1}\cdots \left\| \lambda _\kappa\right\|
			^{d-t_\kappa}}<\infty .
	\end{equation*}
\end{lemma}

\begin{theorem}
	{\rm{\cite{mink}}}\label{th2} Let $\eta(x),$ $x\in \mathbb{R}^d,$ be a
	homogeneous isotropic Gaussian random field with $\mathbf{E}\eta(x)=0.$ If
	Assumptions~{\rm{\ref{ass1}}} and {\rm{\ref{ass2}}} hold, then for $r\to
	\infty$ the finite-dimensional distributions of 
	\begin{equation*}
		X_{r,\kappa}:=r^{(\kappa\alpha
			)/2-d}L^{-\kappa/2}(r)\int_{\Delta(r)}H_\kappa(\eta (x))\,\mathrm{d}x
	\end{equation*}
	converge weakly to the finite-dimensional distributions of 
	\[X_\kappa(\Delta) :=c_2^{\kappa/2}(d,\alpha ) \int_{\mathbb{R}^{d\kappa}}^{{%
					\prime }}K_{\Delta }(\lambda _1+\cdots +\lambda _\kappa)\]
	\begin{equation}  \label{lim1}
		 \times\frac{W(\mathrm{d}%
			\lambda _1)\ldots W(\mathrm{d}\lambda _\kappa)}{\left\| \lambda _1\right\|
			^{(d-\alpha )/2}\cdots \left\| \lambda _\kappa\right\| ^{(d-\alpha )/2}},
	\end{equation}
	where $\int_{\mathbb{R}^{d\kappa}}^{{\prime}}$ denotes the multiple Wiener-It%
	\^{o} integral.
\end{theorem}

\begin{rem}
If $\kappa=1$ the limit $X_\kappa(\Delta)$ is Gaussian. However, for the case  $\kappa>1$  distributional properties of $X_\kappa(\Delta)$  are almost unknown. It was shown that the integrals in (\ref{lim1}) posses absolutely continuous densities, see \cite{Dav,Shi}. The article \cite{Main} proved that these densities are bounded if $\kappa=2.$ Also, for the Rosenblatt distribution, i.e. $\kappa=2$ and a rectangular  $\Delta$, the density and cumulative distribution functions of $X_\kappa(\Delta)$ were studied in \cite{Vei}. An approach to investigate the boundedness of densities of multiple Wiener-It\^{o} integrals was suggested in \cite{Dav}. However, it is difficult to apply this approach to the case $\kappa>2$ as it requires a classification of the peculiarities of general $n$th degree forms.
\end{rem}
\begin{definition}\label{kol}
	Let $Y_1$ and $Y_2$ be arbitrary random variables. The uniform (Kolmogorov)
	metric for the distributions of $Y_1$ and $Y_2$ is defined by the formula 
	\begin{equation*}
		{\rho}\left( Y_1,Y_2\right) =\underset{z\in \mathbb{R}}{\sup }\left| P\left(
		Y_1\leq z\right) -P\left( Y_2\leq z\right) \right| .
	\end{equation*}
\end{definition}

The following result follows from Lemma~1.8~\cite{pet}.

\begin{lemma}
	\label{lem1} If $X,Y$ and $Z$ are arbitrary random variables, then for any $%
	\varepsilon >0:$ 
	\begin{equation*}
		\rho \left( X+Y,Z\right) \leq {\rho }(X,Z)+\rho \left( Z+\varepsilon
		,Z\right) +P\left( \left\vert Y\right\vert \geq \varepsilon \right) .
	\end{equation*}
\end{lemma}

\section{L\'{e}vy concentration functions for $X_{k}(\Delta )$}

\label{sec4} 

In this section, we will investigate some fine properties of probability distributions of Hermite-type random variables. These results will be used to derive upper bounds of ${\rho}\left(X_\kappa(\Delta)+\varepsilon,X_\kappa(\Delta)\right)$ in the next section. The following function from Section~1.5 \cite{pet} will be used in this section.

\begin{definition}\label{levy}
	The L\'{e}vy concentration function of a random variable $X$ is defined by
	\[Q(X,\varepsilon):=\sup\limits_{z\in\mathbb{R}}{\rm{P}}(z<X\leq z+\varepsilon),\quad \varepsilon\geq 0.\]
\end{definition}

We will discuss three important cases, and show how to estimate the L\'{e}vy concentration function in each of them. 

If $X_{k}(\Delta )$ has a bounded probability density function $p_{X_\kappa(\Delta)}\left(\cdot\right),$ then it holds

\begin{equation}\label{bd}
Q\left(X_\kappa(\Delta),\varepsilon\right) \le \varepsilon\sup_{z\in \mathbb R}\, p_{X_\kappa(\Delta)}\left( z\right)\le \varepsilon\, C.
\end{equation}

This inequality is probably the sharpest known estimator of the L\'{e}vy concentration function of $X_{k}(\Delta )$. It is discussed in cases 1 and 2.

\textit{\textbf{Case 1.}} If the Hermite rank of $G(\cdot)$ is equal to $\kappa = 2$ we are dealing with the so-called Rosenblatt-type random variable. It is known that the probability density function of this variable is bounded, consult \cite{Main, Dav, Dav1, anh, leonew} for proofs by different methods. Thus, one can use estimate (\ref{bd}).

\textit{\textbf{Case 2.}} Some interesting results about boundedness of probability density functions of Hermite-type random variables were obtained in \cite{Hu} by Malliavin calculus. To present these results we provide some definitions from Malliavin calculus.

Let $X = \{X(h), h\in L^2(\mathbb{R}^d)\}$ be an isonormal Gaussian process defined on a complete probability space $(\Omega ,\mathcal{F},P)$. Let $\mathbf{S}$ denote the class of smooth random variables of the form $F = f(X(h_1),\dots X(h_n))$, $n \in \mathbb{N}$, where $h_1,\dots, h_n$ are in $L^2(\mathbb{R}^d)$, and $f$ is a function, such that $f$ itself and all its partial derivatives have at most polynomial growth.

The Malliavin derivative $DF$ of $F = f(X(h_1),\dots X(h_n))$ is the $L^2(\mathbb{R}^d)$ valued random variable given by

\[DF = \sum\limits_{i=1}^n \frac{\partial f}{\partial x_i}(X(h_1),\dots X(h_n))h_i.\]

The derivative operator $D$ is a closable operator on $L_2(\Omega)$ taking values in $L_2(\Omega; L^2(\mathbb{R}^d))$.
By iteration one can define higher order derivatives $D^k F \in L_2(\Omega;L^2(\mathbb{R}^d)^{\odot k})$, where $\odot$ denotes the symmetric tensor product. For any integer $k \geq 0$ and any $p \geq 1$ we denote by $\mathbb{D}^{k,p}$ the closure of $\mathbf{S}$ with respect to the norm $\|\cdot\|_{k,p}$ given by

\[ \|F\|^{p}_{k,p} =\sum\limits_{i = 0}^{k}\mathbf{E}\left(\left\|D^i F\right\|_{L^2(\mathbb{R}^d)^{\otimes i}}^p\right).
\]

Let's denote by $\delta$ the adjoint operator of $D$ from a domain in $L_2(\Omega;L^2(\mathbb{R}^d))$ to $L_2(\Omega)$. An element $u \in L_2(\Omega;L^2(\mathbb{R}^d))$ belongs to
the domain of $\delta$ if and only if for any $F \in \mathbb{D}^{1,2}$ it holds
\[ \mathbf{E}[\langle DF, u\rangle] \leq c_u \sqrt{\mathbf{E}[F^2]},\]
where $c_u$ is a constant depending only on $u$.

The following theorem gives sufficient conditions to guarantee boundedness of Hermite-type densities.

\begin{theorem}{\rm\cite{Hu}}
	Let $F \in \mathbb{D}^{2,s}$ such that $\mathbf{E}[|F|^{2q}] < \infty$ and 
\begin{equation}\label{ml}
	\mathbf{E}\left[\left\|DF\right\|^{-2r}_{L^2(\mathbb{R}^d)}\right] < \infty,
\end{equation}
	where $q,r,s >1$ satisfying $\frac{1}{q} +\frac{1}{r}+\frac{1}{s} = 1$.
	
Denote $w = \left\|DF\right\|^{2}_{L^2(\mathbb{R}^d)}$ and  $u = w^{-1} DF$. Then $u \in \mathbb{D}^{1,q^\prime}$ with $q^\prime = \frac{q}{q-1}$ and $F$ has a density given by
	$p_F(x) = \mathbf{E}\left[\mathbf{1}_{F>x}\delta(u) \right]$. Furthermore, $p_F(x)$ is bounded and
	$p_F(x) \leq C_q \|w^{-1}\|_r\|F\|_{2,s}\min(1,|x^{-2}\|F\|^2_{2q}),$
	for any $x\in \mathbb{R}$, where $C_q$ is a constant depending only on $q$.
\end{theorem}

Note, that the Hermite-type random variable $X_\kappa(\Delta)$ does belong to the space $\mathbb{D}^{2,s}, s>1$, and $\mathbf{E}[|X_\kappa(\Delta)|^{2q}] < \infty$ by the hypercontractivity property, see (2.11) in \cite{Hu}. Thus, if  the condition~(\ref{ml}) holds, one can use (\ref{bd}). 

\textit{\textbf{Case 3.}} When there is no information about boundedness of the probability density function, anti-concentration inequalities can be used to obtain estimates of the L\'{e}vy concentration function. 

Let us denote by $I_\kappa(\cdot)$ a multiple Wiener-It\^{o} stochastic integral of order $d\kappa$, i.e. $I_\kappa(f) =\int_{\mathbb{R}^{d\kappa}}^{{%
		\prime }}f(\lambda _1,\cdots,\lambda _\kappa)W(\mathrm{d}%
\lambda _1)\ldots W(\mathrm{d}\lambda _\kappa),$ where $f(\cdot) \in {L}^{s}_{2}(\mathbb{R}^{d\kappa})$. Here ${L}^{s}_{2}(\mathbb{R}^{d\kappa})$ denotes the space of symmetrical functions in ${L}_{2}(\mathbb{R}^{d\kappa})$. Note, that any $F\in L_2(\Omega)$ can be represented as $F = \mathbf{E}(F) + \sum\limits_{q = 1}^\infty I_q(f_q)$, where the functions $f_q$ are determined by $F$. The multiple Wiener-It\^{o} integral $I_q(f_q)$ coincides with the orthogonal projection of $F$ on the $q$-th Wiener chaos associated with $X$. 

The following lemma uses the approach suggested in \cite{NourPol}.

\begin{lemma}\label{sb}
	For any $\kappa \in \mathbb{N}$, $t \in\mathbb{R}$, and $\hat{\varepsilon}>0$ it holds
	\[{\rm{P}}\left(|X_\kappa(\Delta)-t|\leq\hat{\varepsilon}\right)\leq \frac{c_{\kappa}\hat{\varepsilon}^{1/\kappa}}{\left( C\|\hat{K}_{\Delta}\|^{2}_{{L}_{2}(\mathbb{R}^{d\kappa})} + t^2\right)^{1/\kappa}},
	\]
	where $\hat{K}_\Delta(x_1,\dots,x_\kappa) := \frac{K_{\Delta}(x_1 + \dots + x_\kappa)}{\left\| \lambda _1\right\|
		^{(d-\alpha )/2}\cdots \left\| \lambda _\kappa\right\| ^{(d-\alpha )/2}}$ and $c_\kappa$ is a constant that depends on $\kappa$.
\end{lemma}

\begin{proof}
	
	Let $\{e_i\}_{i\in \mathbb{N}}$ be an orthogonal basis of ${L}_{2}(\mathbb{R}^d)$.
	Then, $\hat{K}_\Delta \in {L}_{2}(\mathbb{R}^{d\kappa})$ can be represented as
	\[\hat{K}_{\Delta} = \sum\limits_{(i_1,\dots,i_{\kappa}) \in \mathbb{N}^\kappa} c_{i_1,\dots,i_\kappa}e_{i_1}\otimes\dots\otimes e_{i_\kappa}.
	\]
	For each $n\in \mathbb{N}$, set 
		\[\hat{K}^n_{\Delta} = \sum\limits_{(i_1,\dots,i_\kappa) \in \{1,\dots,n\}^\kappa} c_{i_1,\dots,i_\kappa}e_{i_1}\otimes\dots\otimes e_{i_\kappa}.
	\]
	
	Note, that both $\hat{K}_{\Delta}$ and $\hat{K}^n_{\Delta}$ belong to the space ${L}^{s}_{2}(\mathbb{R}^{d\kappa})$.
	
	By~(\ref{lim1}) it follows that $X_\kappa(\Delta) = c_2^{\kappa/2}(d,\alpha )I_\kappa(\hat{K}_{\Delta})$. Let us denote $X^n_\kappa(\Delta) :=c_2^{\kappa/2}(d,\alpha )I_\kappa(\hat{K}^n_{\Delta})$. 
		
	As $n\rightarrow\infty$, $\hat{K}^n_\Delta\rightarrow \hat{K}_\Delta$ in ${L}_{2}(\mathbb{R}^{d\kappa})$. Thus, $X^n_\kappa(\Delta) \rightarrow X_\kappa(\Delta)$ in $L_2(\Omega ,\mathcal{F},P)$. Hence, there exists a strictly increasing sequence ${n_j}$ for which 
		$X^{n_j}_\kappa(\Delta) \rightarrow X_\kappa(\Delta)$ almost surely as $j \rightarrow \infty$.
		
	It also follows that
	
	\[X^n_\kappa(\Delta) = c_2^{\kappa/2}(d,\alpha )I_\kappa\left(\sum\limits_{(i_1,\dots,i_\kappa)\in\{1,\dots,n\}^\kappa} c_{i_1,\dots,i_\kappa}e_{i_1}\otimes\dots\otimes e_{i_\kappa}\right)
	\]
	\[= c_2^{\kappa/2}(d,\alpha )\sum\limits_{m=1}^\kappa\sum\limits_{\lfrac{1\leq i_1^\prime< \dots<i_m^\prime\leq n}{\kappa_1+\dots+\kappa_m = \kappa}}^n c_{i_1^\prime,\dots,i_m^\prime}^{\kappa_1,\dots,\kappa_m}I_\kappa(e^{\otimes\kappa_1}_{i_1^\prime}\otimes\dots\otimes e^{\otimes\kappa_m}_{i_m^\prime}),\]
	where $\kappa_i\in\mathbb{N}$, $i=1,\dots,m$, $c_{i_1^\prime,\dots,i_m^\prime}^{\kappa_1,\dots,\kappa_m} = \sum\limits_{(i_1,\dots,i_\kappa)\in A_{i_1^\prime,\dots,i_m^\prime}^{\kappa_1,\dots,\kappa_m}}c_{i_1,\dots,i_\kappa},$
	and \\$A_{i_1^\prime,\dots,i_m^\prime}^{\kappa_1,\dots,\kappa_m} :=\{(i_1,\dots,i_\kappa)\in\{1,\dots,n\}^\kappa: \kappa_1 \ \text{indicies } i_l = i_1^\prime,\dots, \kappa_m \ \text{indicies }\\  i_l = i_m^\prime, l = 1,\dots,\kappa\}.$ 
	
	By the It\^{o} formula \cite{iv}:
	\[I_{\kappa_1+\dots+\kappa_m}\left(e_{i_1}^{\otimes\kappa_1}\otimes\dots\otimes e_{i_m}^{\otimes\kappa_m}\right) = \prod\limits_{j=1}^m H_{\kappa_j}\left(\int\limits_{\mathbb{R}^d}e_j(\lambda)W(\mathrm{d}%
		\lambda)\right)=
		\prod\limits_{j=1}^m H_{\kappa_j}(\xi_j),
	\]
	where $\xi_j \sim \mathcal{N}(0,1)$.
	
	Thus, $X^n_\kappa(\Delta)$ can be represented as $X^n_\kappa(\Delta)=U_{n,\kappa}(\hat{\varepsilon}_1,\dots,\hat{\varepsilon}_n),$
	where $U_{n,\kappa}(\cdot)$ is a polynomial of the degree at most $\kappa$. Furthermore,  $X^n_\kappa(\Delta) - t$ is also a polynomial of the degree at most $\kappa$.
	
	Now, applying Carbery-Wright inequality, see Theorem~2.5 \cite{NourPol}, one obtains that there exists a constant $\hat{c}_{\kappa}$ such that for any $n \in \mathbb{N}$ and $\hat{\varepsilon} > 0$
	\[{\rm{P}}\left(|X^n_\kappa(\Delta)-t|\leq\hat{\varepsilon} \left(\mathbf{E}\left(X^n_\kappa(\Delta)-t\right)^2\right)^{\frac{1}{2}}\right)\leq \hat{c}_{\kappa}\hat{\varepsilon}^{1/\kappa}.
	\]
	
	Analogously to \cite{NourPol}, using Fatou's lemma we get 
	
	\[{\rm{P}}\left(|X_\kappa(\Delta)-t|\leq\hat{\varepsilon} \left(\mathbf{E}\left(X_\kappa(\Delta)-t\right)^2\right)^{\frac{1}{2}}\right)\leq\hat{c}_{\kappa}2^{1/\kappa}\hat{\varepsilon}^{1/\kappa} = c_{\kappa}\hat{\varepsilon}^{1/\kappa}.
	\]
	It is known, see (1.3) and (1.5) in \cite{Houd}, that $\mathbf{E}X_\kappa(\Delta) = 0$ and $\mathbf{E}\left(X_\kappa(\Delta)\right)^2 = C\|\hat{K}_{\Delta}\|^2_{{L}_{2}(\mathbb{R}^{d\kappa})}$. Thus, the above inequality can be rewritten as
	\[{\rm{P}}\left(|X_\kappa(\Delta)-t|\leq\hat{\varepsilon}\right)\leq
	\frac{c_{\kappa}\hat{\varepsilon}^{1/\kappa}}{\left(\mathbf{E}\left(X_\kappa(\Delta)-t\right)^2\right)^{\frac{1}{2\kappa}}}=\frac{c_{\kappa}\hat{\varepsilon}^{1/\kappa}}{\left( C\|\hat{K}_{\Delta}\|^{2}_{{L}_{2}(\mathbb{R}^{d\kappa})} + t^2\right)^{1/\kappa}}.
	\]
\end{proof}

The following theorem combines all three cases above and provides an upper-bound estimator of the L\'{e}vy concentration function. 

\begin{theorem}\label{cmb}
	For any $\kappa \in \mathbb{N}$ and an arbitrary positive $\varepsilon$ it holds
	\[Q\left(X_\kappa(\Delta),\varepsilon\right)\leq C\varepsilon^{a},\]
	where the constant $a$ depends on the cases discussed above.
	
\end{theorem}

\begin{proof}
    For cases 1 and 2 it is an immediate corollary of (4.1) and the boundedness of $p_{X_\kappa(\Delta)}(\cdot)$.
	
	For case 3, applying Lemma~\ref{sb} with $t = z + \frac{\varepsilon}{2}$ and $\hat{\varepsilon} = \frac{\varepsilon}{2}$ we get
	
	\[
	Q\left(X_\kappa(\Delta),\varepsilon\right) = \sup\limits_{z\in\mathbb{R}}{\rm{P}}\left(\left|X_\kappa( \Delta) - (z+\frac{\varepsilon}{2})\right|\leq \frac{\varepsilon}{2}\right)\]
	\[ \leq\sup\limits_{z\in\mathbb{R}}\left(\frac{c_{\kappa}\left(\frac{\varepsilon}{2}\right)^{1/\kappa}}{\left(C\|\hat{K}_{\Delta}\|^{2}_{{L}_{2}(\mathbb{R}^{d\kappa})} + \left(z+\frac{\varepsilon}{2}\right)^2\right)^{\frac{1}{2\kappa}}}\right)
	\leq \frac{c_{\kappa}\varepsilon^{1/\kappa}}{\left(2C\|\hat{K}_{\Delta}\|_{{L}_{2}(\mathbb{R}^{d\kappa})}\right)^{\frac{1}{\kappa}}} = C\varepsilon^{1/\kappa}.
	\]
\end{proof}

\begin{rem}
	Notice, that by Definitions~\ref{kol} and \ref{levy}
	
	\[Q(X_\kappa(\Delta),\varepsilon) = \sup\limits_{z\in\mathbb{R}}\left({\rm{P}}(X_\kappa(\Delta)\leq z+\varepsilon)-{\rm{P}}(X_\kappa(\Delta)\leq z)\right)\] \[=\sup\limits_{z\in\mathbb{R}}\left|{\rm{P}}(X_\kappa(\Delta)\leq z)-{\rm{P}}(X_\kappa(\Delta)+\varepsilon\leq z)\right| = {\rho}\left(X_\kappa(\Delta)+\varepsilon,X_\kappa(\Delta)\right).\]
\end{rem}
\section{Rate of convergence}

\label{sec5}

In this section we consider the case of Hermi\-te-type limit distributions in
Theorem~\ref{th2}. The main result describes the rate of convergence of $%
K_{r}$ to $X_{\kappa }(\Delta )$ when $r\rightarrow \infty .$ To prove it we
use some techniques and facts from \cite{bra,mink,anh}.

\begin{theorem}\label{th5}
	Let Assumptions~{\rm{\ref{ass1}}} and {\rm{\ref{ass2}}} hold and $H 
	\mathrm{rank}\,G=\kappa \in \mathbb{N}$.
	
	If $\tau \in \left(-\frac{d - \kappa\alpha}{2},0\right)$ then for any $%
	\varkappa<\frac{a}{2+a}\min\left(\frac{\alpha(d-\kappa\alpha)}{%
		d-(\kappa-1)\alpha},\varkappa_1\right)$ 
	\begin{equation*}
		{\rho}\left( \frac{\kappa!\,K_r}{C_\kappa\,r^{d-\frac{\kappa\alpha}{2}}L^{%
				\frac{\kappa}{2}}(r)},X_\kappa(\Delta)\right)=o (r^{-\varkappa}),\quad
		r\rightarrow \infty ,
	\end{equation*}
	where $a$ is a constant from Theorem~{\rm{\ref{cmb}}}, $C_\kappa$ is defined by {\rm{(\ref{herm})}}, and 
	\begin{equation*}
		\varkappa_1:=\min\left(-2\tau,\frac{1}{\frac{1}{d-2\alpha}+ \dots +\frac{1}{%
				d-\kappa\alpha}+ \frac{1}{d+1-\kappa\alpha}}\right).
	\end{equation*}
	
	If $\tau=0$ then 
	\begin{equation*}
		{\rho}\left( \frac{\kappa!\,K_r}{C_\kappa\,r^{d-\frac{\kappa\alpha}{2}}L^{%
				\frac{\kappa}{2}}(r)},X_\kappa(\Delta)\right)=g^{\frac{2}{3}}(r), \quad
		r\rightarrow \infty.
	\end{equation*}
\end{theorem}

\begin{rem}
	This theorem generalises the result for the Rosenblatt-type case ($\kappa =2$%
	) in \rmfamily\cite{Main} to Hermite-type asymptotics ($\kappa >2$). It
	also relaxes the assumptions on the spectral density used in \cite{Main},
	see Remarks \ref{rem1} - \ref{rem3}.
\end{rem}

\begin{proof} Since $H {\rm rank}\,G=\kappa,$ it follows that $K_r$ can be represented in the space of squared-integrable random variables $L_2(\Omega)$ as
	\[K_r = K_{r,\kappa}+S_r :=\frac{C_\kappa }{\kappa!}
	\int_{\Delta(r)}H_\kappa (\eta (x))\,\mathrm{d}x + \sum_{j\geq \kappa + 1}\frac{C_j}{j!}
	\int_{\Delta(r)}H_j (\eta (x))\,\mathrm{d}x,
	\]
	where $C_j $ are coefficients of the Hermite series (\ref{herm}) of the function $G(\cdot).$
	
	Notice that $\mathbf{E}K_{r,\kappa}=\mathbf{E}S_r=\mathbf{E} X_{\kappa}(\Delta)=0,$ and
	\[X_{r,\kappa}=\frac{\kappa!\,K_{r,\kappa}}{C_\kappa\,r^{d-\frac{\kappa\alpha}{2}}L^{\frac{\kappa}{2}}(r)}.\]
	It follows from Assumption~\ref{ass1} that $|L(u)/u^\alpha|=|B(u)|\le B(0)=1.$ Thus, by the proof of Theorem~4~\cite{mink}, 
	\[
		\mathbf{Var}\, S_{r} \leq |\Delta|^2r^{2d-(\kappa+1)\alpha}\sum_{j\geq \kappa+1}\frac{C_j ^2}{j !}
		\int_0^{diam\left\{ \Delta
			\right\}} z^{-(\kappa+1)\alpha} L^{\kappa+1}\left(rz\right)\psi _{\Delta}(z)dz
	\]
\begin{equation*}
	 \leq|\Delta|^2r^{2d-\kappa\alpha}L^{\kappa}(r)\sum_{j\geq \kappa+1}\frac{C_j ^2}{j !}
		\int_0^{diam\left\{ \Delta
			\right\}} z^{-\kappa\alpha} \frac{L^{\kappa}\left(rz\right)}{L^{\kappa}(r)}  \frac{L\left(rz\right)}{(rz)^{\alpha}}\psi _{\Delta}(z)\,\mathrm{d}z.\label{varq}
\end{equation*}

	We represent the integral in (\ref{varq}) as the sum of two integrals $I_1$ and $I_2$ with the ranges of integration $[0,r^{-\beta_1}]$ and $(r^{-\beta_1},diam\left\{ \Delta
	\right\}]$ respectively, where $\beta_1\in(0,1).$
	
	It follows from Assumption~\ref{ass1} that  $|L(u)/u^\alpha|=|B(u)|\le B(0)=1$ and we can estimate the first integral as 
	\[I_1\le\int_0^{r^{-\beta_1}} z^{-\kappa\alpha} \frac{L^{\kappa}\left(rz\right)}{L^{\kappa}(r)}  \psi _{\Delta}(z)\,\mathrm{d}z
	\le\left(\frac{\sup_{0\le s\le r}s^{\delta/\kappa}L\left(s\right)}{r^{\delta/\kappa}L(r)}\right)^{\kappa}\] \[\times\int_0^{r^{-\beta_1}} z^{-\delta}z^{-\kappa\alpha}  \psi _{\Delta}(z)\,\mathrm{d}z,
	\]
	where $\delta$ is an arbitrary number in $(0,\min(\alpha,d-\kappa\alpha)).$

	By Assumption~\ref{ass1} the function $L\left(\cdot\right)$ is locally bounded. By Theorem~1.5.3 in \cite{bin},  there exists $r_0>0$ and $C>0$ such that for all $r\ge r_0$
	\[\frac{\sup_{0\le s\le r}s^{\delta/2}L\left(s\right)}{r^{\delta/2}L(r)}\le C.\]
	
	Using (\ref{dint}) we obtain
	\[\int_0^{r^{-\beta_1}} z^{-\delta}z^{-\kappa\alpha}  \psi _{\Delta}(z)\,\mathrm{d}z\le\frac{C}{\left| \Delta \right| }\int_{0}^{r^{-\beta_1}}\tau^{d-\kappa\alpha-1-\delta}\,\mathrm{d}\tau=\frac{C\,r^{-\beta_1(d-\kappa\alpha -\delta)}}{(d-\kappa\alpha-\delta)\left|\Delta \right|}.
	\]
	
	Applying Theorem~1.5.3 \cite{bin}  we get
	\[I_2\le
	\frac{\sup_{r^{1-\beta_1}\le s\le r\cdot diam\left\{ \Delta
			\right\}}s^{\delta}L^{\kappa}\left(s\right)}{r^{\delta}L^{\kappa}(r)}\cdot \sup_{r^{1-\beta_1}\le s\le r\cdot diam\left\{ \Delta
		\right\}}\frac{L\left(s\right)}{s^\alpha}\]\[ \int_0^{diam\left\{ \Delta \right\}} z^{-(\delta+\kappa\alpha)}  \psi _{\Delta}(z)\,\mathrm{d}z\le C\cdot o(r^{-(\alpha-\delta)(1-\beta_1)}),
	\]
	when $r$ is sufficiently large.
	
	Notice that by (\ref{par})
	\[\sum_{j\geq \kappa + 1}\frac{C_j ^2}{j !}\le \int_{\mathbb R}G^2(w)\ \phi ( w )\,\mathrm{d}w< +\infty.\]
	
	Hence, for sufficiently large $r$
	\[
	\mathbf{Var}\, S_{r} \leq C\,r^{2d-\kappa\alpha}L^{\kappa}(r)\left( r^{-\beta_1(d-\kappa\alpha-\delta)}
	+
	o\left(r^{-(\alpha-\delta)(1-\beta_1)}\right)\right).
	\]
	Choosing $\beta_1=\frac{\alpha}{d-(\kappa - 1)\alpha}$ to minimize the upper bound we get
	\[
	\mathbf{Var}\, S_{r} \leq C r^{2d-\kappa\alpha}L^{\kappa}(r)r^{-\frac{\alpha(d-\kappa\alpha)}{d-(\kappa - 1)\alpha}+\delta}.
	\]
	
	It follows from Theorem~\ref{cmb} that 
	
	\[{\rho}\left(X_\kappa(\Delta)+\varepsilon,X_\kappa(\Delta)\right)\le C\varepsilon^a.\]
	
	Applying Chebyshev's inequality and Lemma~\ref{lem1} to $X=X_{r,\kappa}$,\linebreak $Y=\frac{\kappa!\,S_r}{C_\kappa\,r^{d-\frac{\kappa\alpha}{2}}L^{\frac{\kappa}{2}}(r)},$ and $Z=X_{\kappa}(\Delta),$ we get
	
	$${\rho}\left( \frac{\kappa!\,K_r}{C_\kappa\,r^{d-\frac{\kappa\alpha}{2}}L^{\frac{\kappa}{2}}(r)},X_\kappa(\Delta)\right)={\rho}\left( X_{r,\kappa}+\frac{\kappa!\,S_r}{C_\kappa\,r^{d-\frac{\kappa\alpha}{2}}L^{\frac{\kappa}{2}}(r)},X_\kappa(\Delta)\right)$$
	\[
	\le {\rho}\left(X_{r,\kappa},X_\kappa(\Delta)\right)+C\left(\varepsilon^a+ \varepsilon^{-2}\,r^{-\frac{\alpha(d-\kappa\alpha)}{d-(\kappa - 1)\alpha}+\delta}\right),
	\]
	for a sufficiently large $r.$
	
	Choosing $\varepsilon:=r^{-\frac{\alpha(d-\kappa\alpha)}{(2+a)(d-(\kappa - 1)\alpha)}}$ to minimize the second term we obtain
	\begin{equation}\label{up11}
		{\rho}\left( \frac{\kappa!\,K_r}{C_\kappa\,r^{d-\frac{\kappa\alpha}{2}}L^{\frac{\kappa}{2}}(r)},X_\kappa(\Delta)\right)\le {\rho}\left(X_{r,\kappa},X_\kappa(\Delta)\right)+C\,r^{\frac{-a\alpha(d-\kappa\alpha)}{(2+a)(d-(\kappa - 1)\alpha)}+\delta}.
	\end{equation}
	
	Applying Lemma~\ref{lem1} once again to $X=X_{\kappa}(\Delta),$ $Y=X_{r,\kappa}-X_{\kappa}(\Delta),$\\ and $Z=X_{\kappa}(\Delta)$ we obtain
	\begin{eqnarray}
		\rho\left(X_{r,\kappa},X_{\kappa}(\Delta) \right) &\leq &\varepsilon_1^a\,C
		+P\left\{ \left|X_{r,\kappa}-X_\kappa(\Delta)\right| \geq \varepsilon_1 \right\}  \notag \\
		&\leq &\varepsilon_1^a\,C +\varepsilon_1^{-2} \mathbf{Var}\left(X_{r,\kappa}-X_\kappa(\Delta)\right).
		\label{38}
	\end{eqnarray}
	
	Now we show how to estimate $\mathbf{Var}\left(X_{r,\kappa}-X_\kappa(\Delta)\right).$
	
	By the self-similarity of Gaussian white noise and  formula (2.1) \cite{Dob}
	\[X_{r,\kappa}\stackrel{\mathcal{D}}{=}
		c^{\frac{\kappa}{2}}_2(d,\alpha)\int_{\mathbb{R}^{\kappa d}}^{{\prime }}K_{\Delta}(\lambda _1+\dots+\lambda _\kappa)Q_r(\lambda _1,\dots,\lambda _\kappa)\]
	\[\times\frac{W(\mathrm{d}\lambda
			_1)\dots W(\mathrm{d}\lambda _\kappa)}{\left\| \lambda _1\right\| ^{(d-\kappa\alpha )/2}\dots
			\left\| \lambda _\kappa\right\| ^{(d-\kappa\alpha )/2}},
	\]
	where
	\[
	Q_r(\lambda _1,\dots,\lambda _\kappa): =r^{\frac{\kappa}{2}(\alpha-d)}L^{-\frac{\kappa}{2}}(r)\
	c_2^{-\frac{\kappa}{2}}(d,\alpha)  \left[ \prod_{i = 1}^{\kappa}\left\| \lambda _i\right\|^{d-\alpha} f\left( \frac{\left\| \lambda _i\right\| }r\right)\right] ^{1/2}.
	\]
	
	Notice that
	\[
	X_\kappa(\Delta) =c^{\frac{\kappa}{2}}_2(d,\alpha )  \int_{\mathbb{R}^{\kappa d}}^{{\prime }}K_{\Delta}(\lambda _1+ \dots +\lambda _\kappa) \frac{W(d\lambda _1) \dots W(d\lambda _\kappa)}{\left\| \lambda
		_1\right\| ^{(d-\alpha )/2} \dots \left\| \lambda _\kappa\right\| ^{(d-\alpha )/2}}.\]
	
	By the isometry property of multiple stochastic integrals
	
	\[R_r:=\frac{\mathbb{E}\left| X_{r,\kappa}-X_\kappa(\Delta)\right|^2}{c_2^{\kappa}(d,\alpha)}\]
	\[=\int_{\mathbb{R}^{\kappa d}}\frac{|K_{\Delta}(\lambda _1+ \dots +\lambda _\kappa)|^2\left(Q_r(\lambda_1,\dots,\lambda_\kappa)-1\right)^2}{\left\| \lambda _1\right\| ^{d-\alpha}\dots\left\| \lambda _\kappa\right\| ^{d-\alpha}}\,\mathrm{d}\lambda
	_1 \dots \,\mathrm{d}\lambda _\kappa. \]

	Let us rewrite the integral $R_r$ as the sum of two integrals $I_3$ and $I_4$ with the integration regions $A(r):=\{(\lambda _1,\dots,\lambda _\kappa)\in\mathbb{R}^{\kappa d}:\ \max\limits_{i = \overline{1,\kappa}}(||\lambda _i||)\le r^\gamma\}$ and $\mathbb{R}^{\kappa d}\setminus A(r)$ respectively, where $\gamma\in(0,1).$ Our intention is to use the monotone equivalence property of regularly varying functions in the regions~$A(r).$
	
	First we consider the case of $(\lambda _1,\dots\lambda _\kappa)\in A(r).$  By Assumption~\ref{ass2} and the inequality \[\left|\sqrt{\prod\limits_{i = 1}^{\kappa}x_i}-1\right|\le \sum\limits_{i = 1}^{\kappa}\left|x^{\frac{\kappa}{2}}_i - 1\right|\] we obtain
	
	\[
	|Q_r(\lambda _1,\dots, \lambda _2)-1| = \left|\sqrt{\prod\limits_{j = 1}^{\kappa}\frac{L\left( \frac{r}{\left\| \lambda_j\right\| }\right)}{L(r)}}-1\right| \le \sum\limits_{j = 1}^{\kappa}\left|\frac{L^{\frac{\kappa}{2}}\left( \frac{r}{\left\| \lambda_j\right\| }\right)}{L^{\frac{\kappa}{2}}(r)}-1\right|.\]

	By Lemma~\ref{lem0}, if $||\lambda _j||\in (1, r^\gamma),$ $j=\overline{1,\kappa},$ then for arbitrary $\delta_1 > 0$ and sufficiently large $r$ we get
	\[\left|1-\frac{L^{\frac{\kappa}{2}}\left( \frac{r}{\left\| \lambda _j\right\| }\right)}{L^{\frac{\kappa}{2}}(r)} \right|=\frac{L^{\frac{\kappa}{2}}\left( \frac{r}{\left\| \lambda _j\right\| }\right)}{L^{\frac{\kappa}{2}}(r)} \cdot\left|1-\frac{L^{\frac{\kappa}{2}}(r)}{L^{\frac{\kappa}{2}}\left( \frac{r}{\left\| \lambda _j\right\| }\right)} \right|\le C\,\frac{L^{\frac{\kappa}{2}}\left( \frac{r}{\left\| \lambda _j\right\| }\right)}{L^{\frac{\kappa}{2}}(r)}g\left( \frac{r}{\left\| \lambda _j\right\|}\right)\]
	\[\times\left\| \lambda _j\right\|^{\delta_1} h_{\tau}(\left\| \lambda _j\right\|)= C\,  \left\| \lambda _j\right\|^{\delta_1} h_{\tau}(\left\| \lambda _j\right\|)g(r)\frac{g\left( \frac{r}{\left\| \lambda _j\right\|}\right)}{g(r)}\left(\frac{L\left( \frac{r}{\left\| \lambda _j\right\| }\right)}{L(r)}\right)^{\frac{\kappa}{2}}.\]
	
	For any positive $\beta_2$ and $\beta_3$, applying Theorem~1.5.6 \cite{bin} to $g(\cdot)$ and $L(\cdot)$ and using the fact that $h_{\tau}\left(\frac{1}{t}\right) = -\frac{1}{t^{\tau}}h(t)$ we obtain
	\[ \left|1-\frac{L^{\frac{\kappa}{2}}\left( \frac{r}{\left\| \lambda _j\right\| }\right)}{L^{\frac{\kappa}{2}}(r)} \right|\le C\,  \left\| \lambda _j\right\|^{\delta_1 + \frac{\kappa\beta_2}{2} + \beta_3 } \left\| \lambda _j\right\|^{- \tau}h_{\tau}(\left\| \lambda _j\right\|)g(r)\]
	\begin{equation}\label{gr1}
		= C\,   \left\| \lambda _j\right\|^{\delta} h_{\tau}\left(\frac{1}{\left\| \lambda _j\right\|}\right)g(r). 
	\end{equation}

	By  Lemma~\ref{lem0} for $||\lambda _j||\le 1,$ $j=\overline{1,\kappa}$, and arbitrary $\delta > 0,$ we obtain
	\begin{equation}\label{gr2}
		\left|1-\frac{L^{\frac{\kappa}{2}}\left( \frac{r}{\left\| \lambda _j\right\| }\right)}{L^{\frac{\kappa}{2}}(r)} \right| \le C\,   \left\| \lambda _j\right\|^{-\delta} h_{\tau}\left(\frac{1}{\left\| \lambda _j\right\|}\right)g(r).
	\end{equation}
	
	Hence, by  (\ref{gr1}) and (\ref{gr2})
	\[
	|Q_r(\lambda _1,\dots \lambda _\kappa)-1|^2 \le k \sum\limits_{j = 1}^{\kappa}\left|\frac{L^{\frac{\kappa}{2}}\left( \frac{r}{\left\| \lambda_j\right\| }\right)}{L^{\frac{\kappa}{2}}(r)}-1\right|^{2}\] 
	\[\le C\sum\limits_{j = 1}^{\kappa} h^{2}_{\tau}\left(\frac{1}{\left\| \lambda _j\right\|}\right)g^{2}(r) \max\left(\left\| \lambda _j\right\|^{-\delta}, \left\| \lambda _j\right\|^{\delta}\right),
	\]
	for $(\lambda _1,\dots \lambda _\kappa)\in A(r)$.

	Notice, that in the case $\tau = 0$ for any  $\delta>0$ there exists $C>0$ such that \linebreak $h_0(x)=\ln(x)<Cx^{\delta},\,x \ge 1$, and $h_0(x)=\ln(x)<Cx^{-\delta},\,x < 1$. Hence, by Lemma~\ref{finint} for $-\tau \le \frac{d -\kappa\alpha}{2}$ we get
	\[\int\limits_{A(r)\cap [0,1]^{\kappa d}}\frac{ h^{2}_{\tau}\left(\frac{1}{\left\| \lambda _j\right\|}\right)\max\left(\left\| \lambda_j\right\|^{-\delta},\left\|\lambda_j\right\|^{\delta}\right)\left|K_{\Delta}\left(\sum\limits_{i=1}^{\kappa}\lambda_i\right)\right|^2\,\mathrm{d}\lambda_1 \dots \mathrm{d}\lambda _\kappa }{\left\| \lambda _1\right\| ^{d-\alpha}\dots\left\| \lambda _\kappa\right\| ^{d-\alpha}} < \infty.\]

	Therefore, we obtain for sufficiently large $r$
	\[I_3\le C\,g^{2}(r) \sum\limits_{j = 1}^{\kappa}\int\limits_{A(r)\cap \mathbb{R}^{\kappa d}}\frac{ h^{2}_{\tau}\left(\frac{1}{\left\| \lambda _j\right\|}\right)\cdot \max\left(\left\| \lambda _j\right\|^{-\delta}, \left\| \lambda _j\right\|^{\delta}\right) }{\left\| \lambda _1\right\| ^{d-\alpha}\dots\left\| \lambda _\kappa\right\| ^{d-\alpha}}\]
	\[
		\times|K_{\Delta}(\lambda _1+\dots \lambda _\kappa)|^2\,\mathrm{d}\lambda
		_1 \dots \mathrm{d}\lambda _\kappa\le C\,g^{2}(r) \int\limits_{A(r)\cap \mathbb{R}^{\kappa d}}\frac{ h^{2}_{\tau}\left(\frac{1}{\left\| \lambda _1\right\|}\right) }{\left\| \lambda _1\right\| ^{d-\alpha}\dots\left\| \lambda _\kappa\right\| ^{d-\alpha}}\] 
	\begin{equation}\label{del}
		\times\max\left(\left\| \lambda _1\right\|^{-\delta}, \left\| \lambda _1\right\|^{\delta}\right)|K_{\Delta}(\lambda _1+\dots \lambda _\kappa)|^2\,\mathrm{d}\lambda
		_1 \dots \mathrm{d}\lambda _\kappa \le C\,g^{2}(r).
	\end{equation}

	It follows from Assumption~\ref{ass2} and the specification  of the estimate (23)  in the proof of Theorem~5 \cite{mink} that for each positive $\delta$ there exists $r_0>0$ such that for all $r\ge r_0,$   $(\lambda _1,\dots,\lambda _\kappa)\in B_{(1,\mu_2,\dots,\mu_\kappa)}=\{(\lambda _1,\dots,\lambda _\kappa)\in\mathbb{R}^{\kappa d}: ||\lambda_j||\le 1,\ \mbox{if}\ \mu_j=-1,\ \mbox{and } ||\lambda_j||> 1, \ \mbox{if}\ \mu_j=1, j=\overline{1,k}\},$ and $\mu_j\in \{-1,1\},$ it holds
	\[
	\frac{|K_{\Delta}(\lambda _1+\dots +\lambda _\kappa)|^2\left(Q_r(\lambda _1,\dots\lambda _\kappa)-1\right)^2}{\left\| \lambda _1\right\| ^{d-\alpha}\dots
		\left\| \lambda _\kappa\right\| ^{d-\alpha}}\le \frac{C\, |K_{\Delta}(\lambda _1 +\dots+\lambda _\kappa)|^2}{\left\| \lambda _1\right\| ^{d-\alpha}
		\dots\left\| \lambda _\kappa\right\| ^{d-\alpha}}\]
	\[+C\,\frac{|K_{\Delta}(\lambda _1+\dots+\lambda _\kappa)|^2}{\left\| \lambda _1\right\| ^{d-\alpha-\delta}
		\left\| \lambda _2\right\| ^{d-\alpha-\mu_2\delta}\dots\left\| \lambda _\kappa\right\| ^{d-\alpha-\mu_\kappa\delta}}.
	\]
	
	Since the integrands are non-negative, we can estimate $I_4$ as it is shown below
	\[ I_4\le \kappa\int\limits_{\mathbb{R}^{(\kappa-1)d}}\int\limits_{||\lambda _1||> r^\gamma}\frac{|K_{\Delta}(\lambda _1+\dots+\lambda _\kappa)|^2\left(Q_r(\lambda_1,\dots,\lambda_\kappa)-1\right)^2\mathrm{d}\lambda
		_1 \dots\mathrm{d}\lambda _\kappa}{\left\| \lambda _1\right\| ^{d-\alpha}\dots\left\| \lambda _\kappa\right\| ^{d-\alpha}}\]
	\[ \le C\int\limits_{\mathbb{R}^{(\kappa-1)d}}\int\limits_{||\lambda _1||> r^\gamma}\frac{|K_{\Delta}(\lambda _1+\dots+\lambda _2)|^2\,\mathrm{d}\lambda
		_1 \dots\mathrm{d}\lambda _\kappa}{\left\| \lambda _1\right\| ^{d-\alpha}\dots\left\| \lambda _\kappa\right\| ^{d-\alpha}}\]
	\[+\,  C\sum\limits_{\lfrac{\mu_i\in\{0,1,-1\}}{i \in \overline{1,\kappa}}} \int\limits_{\mathbb{R}^{(\kappa-1)d}}\int\limits_{||\lambda _1||> r^\gamma}\frac{|K_{\Delta}(\lambda _1+\dots +\lambda _\kappa)|^2\mathrm{d}\lambda_1 \dots\mathrm{d}\lambda_\kappa}{\left\| \lambda _1\right\| ^{d-\alpha-\delta}\left\| \lambda _2\right\| ^{d-\alpha-\mu_2\delta}\dots\left\| \lambda _\kappa\right\| ^{d-\alpha-\mu_\kappa\delta}}\]
	\[
	\le C\max_{\lfrac{ \mu_i\in\{0,1,-1\}}{i \in \overline{2,\kappa} }}\int\limits_{\mathbb{R}^{(\kappa-1)d}}\int\limits_{||\lambda _1||> r^\gamma}|K_{\Delta}(\lambda_1 +\dots +\lambda _\kappa)|^2
	\]
	\begin{equation}\label{midest}
		\times\frac{\mathrm{d}\lambda_1 \dots\mathrm{d}\lambda_\kappa}{\left\| \lambda _1\right\| ^{d-\alpha-\delta}\left\|\lambda _2\right\| ^{d-\alpha-\mu_2\delta}\dots\left\| \lambda _\kappa\right\| ^{d-\alpha-\mu_\kappa\delta}}.
	\end{equation}
	Replacing $\lambda_1 + \lambda_2$ by $u$ we obtain
	\[I_4\le C\max_{\lfrac{ \mu_i\in\{0,1,-1\}}{i \in \overline{2,\kappa} }}\int\limits_{\mathbb{R}^{(\kappa-1)d}}\int\limits_{||\lambda _1||> r^\gamma}\frac{|K_{\Delta}(u + \lambda _3+\dots +\lambda _\kappa)|^2}{\left\| \lambda _1\right\| ^{d-\alpha-\delta}\left\|u - \lambda _1\right\| ^{d-\alpha-\mu_2\delta}}\]
	\[
	\times\frac{\mathrm{d}\lambda_1\mathrm{d}u\mathrm{d}\lambda_3 \dots\mathrm{d}\lambda_\kappa}{\left\| \lambda _3\right\| ^{d-\alpha-\mu_3\delta}\dots\left\| \lambda _\kappa\right\| ^{d-\alpha-\mu_\kappa\delta}}\le C\max\limits_{\lfrac{ \mu_i\in\{0,1,-1\}}{i \in \overline{2,\kappa} }} \int\limits_{\mathbb{R}^{(\kappa-1)d}}\frac{1}{\left\|u\right\| ^{d-2\alpha-(\mu_2+1)\delta}}\]
	\[\times\frac{|K_{\Delta}(u + \lambda _3+\dots +\lambda _\kappa)|^2}{\left\| \lambda _3\right\| ^{d-\alpha-\mu_3\delta}\dots\left\| \lambda _\kappa\right\| ^{d-\alpha-\mu_\kappa\delta}}\int\limits_{\|\lambda _1\|> \frac{r^\gamma}{\|u\|}}\frac{\mathrm{d}\lambda_1 \mathrm{d}u\,\mathrm{d}\lambda_3 \dots\mathrm{d}\lambda_\kappa}{\left\| \lambda _1\right\| ^{d-\alpha-\delta}\left\|\frac{u}{\left\|u\right\|} - \lambda _1\right\| ^{d-\alpha-\mu_2\delta}}.\]
	Taking into account that for $\delta\in (0,\min(\alpha,{d}/{\kappa}-\alpha))$
	\[\sup_{u\in\mathbb{R}^{d}\setminus\{0\}}\int_{\mathbb{R}^{d}}\frac{\,\mathrm{d}\lambda
		_1}{\left\| \lambda _1\right\| ^{d-\alpha-\delta}\left\|\frac{u}{\|u\|}- \lambda _1\right\| ^{d-\alpha-\mu_2\delta}}\le C,\]
	we obtain
	\[I_4\le C\max_{\lfrac{ \mu_i\in\{0,1,-1\}}{i \in \overline{3,\kappa} }}\int\limits_{\mathbb{R}^{(\kappa-2)d}}\left[\max_{\mu_2\in\{0,1,-1\}}\int\limits_{||u||\le r^\gamma_0}\frac{|K_{\Delta}(u + \lambda _3+\dots +\lambda _\kappa)|^2}{\left\|u\right\| ^{d-2\alpha-(\mu_2+1)\delta}}\right.\]
	\[\times\frac{\mathrm{d}\lambda_3 \dots\mathrm{d}\lambda_\kappa}{\left\| \lambda _3\right\| ^{d-\alpha-\mu_3\delta}\dots\left\| \lambda _\kappa\right\| ^{d-\alpha-\mu_\kappa\delta}}\left.\int\limits_{||\lambda _1||> r^{\gamma-\gamma_0}}\frac{\mathrm{d}\lambda_1 \mathrm{d}u}{\left\| \lambda _1\right\| ^{d-\alpha-\delta}\left\|\frac{u}{\left\|u\right\|} - \lambda _1\right\| ^{d-\alpha-\mu_2\delta}}\right.\]
	\[+\left.\max_{\mu_i\in\{0,1,-1\}}\int\limits_{||u||> r^\gamma_0}\frac{|K_{\Delta}(u + \lambda _3+\dots +\lambda _\kappa)|^2\mathrm{d}u\,\mathrm{d}\lambda_3 \dots\mathrm{d}\lambda_\kappa}{\left\|u\right\| ^{d-2\alpha-(\mu_2+1)\delta}\left\| \lambda _3\right\| ^{d-\alpha-\mu_3\delta}\dots\left\| \lambda _\kappa\right\| ^{d-\alpha-\mu_\kappa\delta}}\right],\]
	where $\gamma_0\in (0,\gamma).$
	
	By Lemma~\ref{finint}, there exists $r_0>0$ such that for all $r\ge r_0$ the first summand is bounded by
	\[C\max_{\mu_2\in\{0,1,-1\}}\int\limits_{||u||\le r^\gamma_0}\frac{|K_{\Delta}(u + \lambda _3+\dots +\lambda _\kappa)|^2\mathrm{d}u\mathrm{d}\lambda_3 \dots\mathrm{d}\lambda_\kappa}{\left\|u\right\| ^{d-2\alpha-(\mu_2+1)\delta}\left\| \lambda _3\right\| ^{d-\alpha-\mu_3\delta}\dots\left\| \lambda _\kappa\right\| ^{d-\alpha-\mu_\kappa\delta}}\]
	\[\times\int\limits_{||\lambda _1||> r^{\gamma-\gamma_0}}\frac{\mathrm{d}\lambda_1} {\left\| \lambda _1\right\| ^{2d-2\alpha-\delta-\mu_2\delta}}\le Cr^{-(\gamma-\gamma_0)(d-2\alpha-2\delta)}.\]
	
	Therefore, for sufficiently large $r,$
	\[I_4\le Cr^{-(\gamma-\gamma_0)(d-2\alpha-2\delta)}\]
	\[+C\max_{\lfrac{ \mu_i\in\{0,1,-1\}}{i \in \overline{3,\kappa} }} \int\limits_{\mathbb{R}^{(\kappa-2)d}}\int\limits_{||u||> r^{\gamma_0}}\frac{|K_{\Delta}(u + \lambda _3+\dots +\lambda _\kappa)|^2\mathrm{d}u\mathrm{d}\lambda_3 \dots\mathrm{d}\lambda_\kappa}{\left\|u\right\| ^{d-2\alpha-2\delta}\left\| \lambda _3\right\| ^{d-\alpha-\mu_3\delta}\dots\left\| \lambda _\kappa\right\| ^{d-\alpha-\mu_\kappa\delta}}.
	\]
	
	Notice that the second summand here coincides with \rm(\ref{midest}) if $\kappa$ is replaced by $\kappa-1$. Thus, we can repeat the above procedure $\kappa-2$ more times and get the result
	\[I_4\le Cr^{-(\gamma-\gamma_0)(d-2\alpha-2\delta)}+ \dots + Cr^{-(\gamma_{\kappa-3}-\gamma_{\kappa-2})(d-\kappa\alpha-\kappa\delta)}
	\]
	\begin{equation}\label{+}
		+ C\int_{\|u\|> r^{\gamma_{\kappa-2}}}
		\frac{|K_{\Delta}(u)|^2\,\mathrm{d}u}{\left\| u\right\| ^{d-\kappa\alpha-\kappa\delta}},
	\end{equation}
	where $\gamma>\gamma_0>\gamma_1>\dots>\gamma_{\kappa-2}.$
	
	By the spherical $L_2$-average decay rate of the Fourier transform \cite{bra} for $\delta<d+1-\kappa\alpha$ and sufficiently large $r$ we get the following estimate of the integral in (\ref{+})
	\[\int_{\|u\|> r^{\gamma_{\kappa-2}}}
	\frac{|K_{\Delta}(u)|^2\,\mathrm{d}u}{\left\| u\right\| ^{d-\kappa\alpha-\kappa\delta}}\le 
	C\int_{z> r^{\gamma_{\kappa-2}}}\int_{S^{d-1}}
	\frac{|K_{\Delta}(z\omega)|^2}{z^{1-\kappa\alpha-\kappa\delta}}\,\mathrm{d}\omega \mathrm{d}z\]
	\[\le C\int_{z> r^{\gamma_{\kappa-2}}} \frac{\,\mathrm{d}z}{z^{d+2-\kappa\alpha-\kappa\delta}}=
	C\,r^{-\gamma_{\kappa-2}(d+1-\kappa\alpha-\kappa\delta)} \]
	\begin{equation}\label{uppp}
		= C\,r^{-\left(\gamma_{\kappa-2}-\gamma_{\kappa-1}\right) (d+1-\kappa\alpha-\kappa\delta)},
	\end{equation}
	where $S^{d-1}:=\{x\in \mathbb{R}^{d}:\left\Vert x\right\Vert =1\}$  is a sphere
	of radius 1 in $\mathbb{R}^d$ and $\gamma_{\kappa-1}=0.$
	
	Now let's consider the case $\tau<0$. In this case by Theorem~1.5.6 \cite{bin} for any $\delta>0$ we can estimate $g(r)$ as follows 
	\begin{equation}\label{less0}
		g(r)\le C\,r^{\tau+\delta}.
	\end{equation}
	Combining estimates (\ref{up11}), (\ref{38}), (\ref{del}), (\ref{+}), (\ref{uppp}),(\ref{less0}) and choosing\\ $\varepsilon_1:=r^{-\beta},$  we obtain
	\[{\rho}\left( \frac{\kappa!\,K_r}{C_\kappa\,r^{d-\frac{\kappa\alpha}{2}}L^{\frac{\kappa}{2}}(r)},X_\kappa(\Delta)\right)\le C\left(r^{-\frac{a\alpha(d-\kappa\alpha)}{(2+a)(d-(\kappa-1)\alpha)}+\delta}+ r^{-a\beta} +
	r^{2\tau + 2\delta + 2\beta}\right.\]
	\[ + \left.r^{-(\gamma-\gamma_0)(d-2\alpha-2\delta)+2\beta}+ \dots + r^{-(\gamma_{\kappa-3}-\gamma_{\kappa-2})(d-\kappa\alpha-\kappa\delta)+2\beta}\right.\]
	\[\left.+ r^{-\left(\gamma_{\kappa-2}-\gamma_{\kappa-1}\right) (d+1-\kappa\alpha-\kappa\delta)+2\beta}\right).
	\]

	Therefore, for any  $\tilde\varkappa_1\in (0,\frac{2+a}{a}\varkappa_0)$ one can choose a sufficiently small $\delta>0$ such that 
	\begin{equation}\label{bou}
		{\rho}\left( \frac{\kappa!\,K_r}{C_\kappa\,r^{d-\frac{\kappa\alpha}{2}}L^{\frac{\kappa}{2}}(r)},X_\kappa(\Delta)\right)\le Cr^\delta\left(r^{-\frac{a\alpha(d-\kappa\alpha)}{(2+a)(d-(\kappa-1)\alpha)}}+ r^{-\frac{a\tilde\varkappa_1}{2+a}}\right),
	\end{equation}
	where  \[\varkappa_0:=\sup_{\stackrel{1>\gamma>\gamma_0>\dots>\gamma_{\kappa-1}=0}{\beta>0}}\min\left(a\beta,-2\tau-2\beta,(\gamma-\gamma_0)(d-2\alpha)-2\beta, \dots , \right.\]
	\[\left.(\gamma_{\kappa-3}-\gamma_{\kappa-2})(d-\kappa\alpha)-2\beta, \left(\gamma_{\kappa-2}-\gamma_{\kappa-1}\right) (d+1-\kappa\alpha)-2\beta\right).
	\]
	\begin{lemma}\label{max}
		Let $\mathbf{\Gamma} = \left\{\gamma = (\gamma_1,\dots,\gamma_{n+1}) \left|b=\gamma_0>\gamma_1>\dots>\gamma_{n+1}=0\right.\right\}$ and $\mathbf{x} = (x_0,\dots,x_n)\in \mathbb{R}^{n+1}_{+}$ be some fixed vector. 
		
		The function $G(\gamma) = \min\limits_{i}\left(\gamma_i-\gamma_{i+1}\right)x_i$ 
		reaches its maximum at \\ $\bar{\gamma}=(\bar{\gamma}_0,\dots,\bar{\gamma}_{n+1})\in \mathbf{\Gamma}$ such that for any $0 \le i \le n$ it holds
		\begin{equation}\label{midpr}
			\left(\bar{\gamma_i}-\bar{\gamma}_{i+1}\right)x_i = \left(\bar{\gamma}_{i+1} -\bar{\gamma}_{i+2}\right)x_{i+1}.
		\end{equation}
		
	\end{lemma}
	
	\begin{proof}
		Let us show that any deviation of $\gamma$ from $\bar{\gamma}$ leads to a smaller result.
		Consider a vector $\hat{\gamma}$ such that for some $i \in \overline{1,n}$ and some $\varepsilon>0$ the following relation is true 
		\[\hat{\gamma_i}-\hat{\gamma}_{i+1} = \bar{\gamma_i}-\bar{\gamma}_{i+1} + \varepsilon.\] 
		Since $\sum\limits_{i=0}^{n}\hat{\gamma_i}-\hat{\gamma}_{i+1}= \hat{\gamma_0}-\hat{\gamma}_{n+1} = b$ we can conclude that there exist some $j \neq i,\, j \in \overline{1,n},$ and $\varepsilon_1>0$ such that $\hat{\gamma_j}-\hat{\gamma}_{j+1} = \bar{\gamma_j}-\bar{\gamma}_{j+1} - \varepsilon_1$. 
		
		Obviously, in this case 
		\[G(\hat{\gamma}) \le \left(\hat{\gamma_j}-\hat{\gamma}_{j+1}\right)x_j = \left(\bar{\gamma_j}-\bar{\gamma}_{j+1} - \varepsilon_1\right)x_j =\left(\bar{\gamma_j}-\bar{\gamma}_{j+1}\right)x_j - \varepsilon_1 x_j\]
		Since $\varepsilon_1 >0$ and $x_j >0$ it follows from \rm(\ref{midpr}) that
		\[G(\hat{\gamma}) \le \left(\bar{\gamma_j}-\bar{\gamma}_{j+1}\right)x_j - \varepsilon_1 x_j < \left(\bar{\gamma_j}-\bar{\gamma}_{j+1}\right)x_j = G(\bar{\gamma}).\]
		So it's clearly seen that any deviation from $\bar{\gamma}$ will yield a smaller result.
	\end{proof}
	
	Note, that for fixed $\gamma\in(0,1)$ by Lemma~\ref{max}
	\[\sup_{\gamma>\gamma_0>\dots>\gamma_{\kappa-1}=0}\min\left((\gamma-\gamma_0)(d-2\alpha), \dots ,(\gamma_{\kappa-3}-\gamma_{\kappa-2})(d-\kappa\alpha)\right.,\]
	\[\left. \left(\gamma_{\kappa-2}-\gamma_{\kappa-1}\right) (d+1-\kappa\alpha)\right)=\frac{\gamma}{\frac{1}{d-2\alpha}+ \dots +\frac{1}{d-\kappa\alpha}+ \frac{1}{d+1-\kappa\alpha}}
	\]
	and 
	\[\sup_{\gamma\in(0,1)}\frac{\gamma}{\frac{1}{d-2\alpha}+ \dots +\frac{1}{d-\kappa\alpha}+ \frac{1}{d+1-\kappa\alpha}}=\frac{1}{\frac{1}{d-2\alpha}+ \dots +\frac{1}{d-\kappa\alpha}+ \frac{1}{d+1-\kappa\alpha}}.\]
	
	Note that  $\varkappa_0=\sup_{\beta>0}\min\left(a\beta,
	\varkappa_1-2\beta\right)=\frac{a\varkappa_1}{2+a}.$
	
	Finally, from (\ref{bou}) for $\tilde\varkappa_1<\varkappa_1$ the first statement of the theorem follows. 
	
	Now let's consider the case $\tau = 0$.
	In this case by Theorem~1.5.6 \cite{bin} for any $s>0$ and sufficiently large $r$
	\begin{equation}\label{eq0}
		g(r)>r^{-s}.
	\end{equation}
	Combining estimates (\ref{up11}), (\ref{38}), (\ref{del}), (\ref{+}), (\ref{uppp}), replacing all powers of $r$ for $g^2(r)$ using (\ref{eq0}), and choosing $\varepsilon_1:=g^{\beta}(r),\,\beta\in(0,1)$  we obtain
	
	\[{\rho}\left( \frac{\kappa!\,K_r}{C_\kappa\,r^{d-\frac{\kappa\alpha}{2}}L^{\frac{\kappa}{2}}(r)},X_\kappa(\Delta)\right)\le C\left(g^2(r)+g^{\beta}(r)+ g^{2-2\beta}\right).\]
	
	Since $\sup\limits_{\beta\in(0,1)}\min(2,\beta,2-2\beta)=\frac{2}{3},$
	it follows that
	\[{\rho}\left( \frac{\kappa!\,K_r}{C_\kappa\,r^{d-\frac{\kappa\alpha}{2}}L^{\frac{\kappa}{2}}(r)},X_\kappa(\Delta)\right)\le Cg^{\frac{2}{3}}(r).\]
	This proves the second statement of the theorem.
\end{proof}

\begin{rem}
The upper bound on the rate of convergence in Theorem~\ref{th5} is given by  explicit formulae that are  easy to evaluate and analyse. For example, for fixed values of $\alpha$ and $\kappa$ it is simple to see that the upper bound for $\varkappa$ approaches $\frac{a}{2+a}\min\left(\alpha,-2\tau\right),$ when  $d \to +\infty.$ For fixed values of $d$ and $\kappa$ the upper bound  for $\varkappa$ is of the order of magnitude of $O(d-\kappa\alpha)$, when $\alpha \to d/\kappa.$ This result is expected as the value $\alpha = d/\kappa$ corresponds to the boundary where a phase transition between short- and long-range dependence occurs. 
\end{rem}

\section{Conclusion}
\label{sec6}

The rate of convergence to Hermite-type limit distributions in
non-central limit theorems was investigated. The results were obtained under
rather general assumptions on the spectral densities of the considered
random fields, that weaken the assumptions used in \cite{Main}. Similar to 
\cite{Main}, the direct probabilistic approach was used, which has, in our
view, an independent interest as an alternative to the methods in \cite%
{Bre1, NourPec1, NourPec2}.
Additionally, some fine properties of the probability distributions of Hermite-type random variables were investigated. Some special cases when their probability density functions
are bounded were discussed. New anti-concentration inequalities were derived for L\'{e}vy concentration functions. 


\end{document}